\documentclass[a4paper,12pt]{article}
\usepackage{comment}
\usepackage{cite}
\usepackage{amsmath}
\usepackage{amssymb}
\usepackage{amsfonts}
\usepackage[T1]{fontenc}
\usepackage[utf8]{inputenc}
\usepackage{graphicx}
\usepackage{fancyhdr}
\usepackage{float}
\usepackage{xcolor}
\usepackage{authblk}
\usepackage{mathrsfs}
\usepackage{empheq}
\usepackage[hyphens]{url}
\usepackage{hyperref}
\usepackage[]{breakurl}
\usepackage[]{amsthm}
\usepackage{tikz}
\usetikzlibrary{decorations.markings}

\usepackage{geometry}
\newtheorem{theorem}{Theorem}

\begin{document}

\title{Three expressions of the $n$-th prime number: discrete sieving, spectral analysis and probabilistic dynamics}

\author[$\dagger$]{Jean-Christophe {\sc Pain}\footnote{jean-christophe.pain@cea.fr}\\
\small
$^1$CEA, DAM, DIF, F-91297 Arpajon, France\\
$^2$Universit\'e Paris-Saclay, CEA, Laboratoire Mati\`ere en Conditions Extr\^emes,\\ 
F-91680 Bruy\`eres-le-Ch\^atel, France
}

\date{}

\maketitle

\begin{abstract}
The search for a closed-form expression of the $n$-th prime number, $p_n$, has long oscillated between the rigid determinism of analytic functions and the apparent randomness of local distributions. This paper explores three different approaches to $p_n$. The first one formalizes an analytical identity for $p_{n}$ based on a harmonic summation filtered by a M\"obius-derived coprimality indicator. Unlike Gandhi's 1971 identity, which employs a geometric density and logarithmic extraction, this formula operates through a discrete summation over the range defined by Bertrand's postulate. In the second one, we refine the ``harmonic resonance'' model, which posits that primes emerge as spectral nodes from von Mangoldt oscillations. Third, we adopt a ``survival dynamics'' approach, inspired by Mertens' theorems, treating prime spacing as an evolutionary growth process. By bridging these perspectives, we offer a comprehensive framework for understanding the transition from asymptotic trends to discrete arithmetic realities.
\end{abstract}

\section{Introduction: the asymptotic baseline}\label{sec1}

The distribution of prime numbers is fundamentally anchored by the prime-number theorem, which dictates that $p_n \sim n \ln n$ \cite{Gauss1792,Hardy1938}. However, the smooth curve provided by the logarithmic integral $\text{Li}(x)$:
$$
\mathrm {Li} (x)=\mathrm {li} (x)-\mathrm {li} (2)=\int _{2}^{x}{\frac {\mathrm {d} t}{\ln(t)}},
$$
where $\mathrm{li}(x)$ is the Cauchy principal value
$$
\mathrm{li}(x)=\lim _{\varepsilon \to 0}\left(\int _{0}^{1-\varepsilon }{\frac {\mathrm {d} t}{\ln(t)}}+\int _{1+\varepsilon }^{x}{\frac {\mathrm {d} t}{\ln(t)}}\right),
$$
represents only a mean density distribution \cite{Rosser1962}. In reality, the sequence of primes $\{p_n\}$ exhibits ``chatter''-local fluctuations that reflect the complex underlying structure of the integers. To move beyond mere approximation, one must account for the discrete ``noise'' of the distribution. This noise is not stochastic, but rather the result of a spectral interference on one hand, and a cumulative probability of survival on the other. While Gandhi's formula \cite{Gandhi1971} uses the inclusion-exclusion principle and logarithms, it is often compared to other ``exact'' prime formulas in academic literature \cite{VandenEynden1972}. Willans's formula \cite{Willans1964} is the famous predecessor to Gandhi's work, using Wilson's theorem. Dudley offers a historical survey \cite{Dudley1969} that places Gandhi's work in the context of the long search for an analytical representation of $p_n$. Following Golomb's 1974 ``demystification'' \cite{Golomb1974a}, researchers looked for ways to make the formula more general or computationally interesting, Ernvall offers a variation of Gandhi's identity that attempts to reduce the complexity of the alternating sums \cite{Ernvall1975}. As noted in the appendix, Golomb showed that the formula represents the probability in binary form, which is essentially an analytical version of the sieve of Eratosthenes \cite{Meyre}. Golomb further proposed a direct algorithmic realization of the sieve itself, independent of Gandhi’s identity, showing that the Eratosthenes process can be fully encoded in a purely analytical framework \cite{Golomb1974b}. The quest for analytical representations remains a vibrant field of study, with recent contributions continuing to refine the structural properties of these identities \cite{Atanassov2021}. Ribenboim's treatise provides an extensive discussion on Gandhi's formula \cite{Ribenboim2011}, highlighting its theoretical significance despite a computational complexity constrained by $2^n$ terms in the summation. By explicitly mapping the analytical structure of the formula onto the classical sieve of Eratosthenes, the work of Regimbal \cite{Regimbal1975} remains a pivotal modern reference, as it refines the nexus between the continuous sieving process and the discrete identification of the next prime $p_{n+1}$. 

Building upon these foundations, this paper presents three distinct conceptual frameworks. The first model formalizes, in section \ref{sec2}, a discrete sieving identity via M\"obius-derived coprimality indicators, acting as a digital gating mechanism through iterative filtering. The second approach, presented in section \ref{sec3}, shifts toward spectral analytic number theory, interpreting prime distribution as a phenomenon of harmonic resonance resulting from wave interference. Finally, we propose, in section \ref{sec4}, a third model based on survival dynamics, where prime emergence is treated as a cumulative growth process driven by probabilistic evolutionary pressure. The first formula is a rigorous arithmetic identity. The second and third approaches are not exact formulas but phenomenological models that aim at capturing complementary structural aspects of the prime sequence.

\section{First formula: discrete sieving via M\"obius-derived coprimality indicators}\label{sec2}

In this section, we formalize an analytical representation of the $(n+1)$-th prime through a discrete sieving identity where the harmonic sum serves as a certificate of uniqueness.

\begin{theorem}
Let $(p_n)_{n\ge1}$ denote the sequence of prime numbers and $P_n = \prod_{i=1}^n p_i$ the $n$-th primorial. Let $\chi_n(m)$ be the arithmetical filter defined by:
\begin{equation*}
\chi_n(m) = \sum_{d|\gcd(m, P_n)} \mu(d),
\end{equation*}
where $\mu$ is the M\"obius function \cite{Mobius1832}. Then, the $(n+1)$-th prime is the unique integer $m > 1$ in the range $[1, 2p_n]$ such that $\chi_n(m)=1$, which can be expressed as:
\begin{equation*}
p_{n+1} = \min \{ m \in \mathbb{N} : m > 1 \text{ and } \chi_n(m) = 1 \}.
\end{equation*}
Furthermore, the following harmonic identity holds:
\begin{equation*}
\left\lfloor \sum_{m=1}^{2p_n} \frac{\chi_n(m)}{m} \right\rfloor = 1.
\end{equation*}
\end{theorem}

\begin{proof}
The M\"obius sum $\sum_{d|k} \mu(d)$ is $1$ if $k=1$ and $0$ otherwise. Thus, $\chi_n(m) = 1$ if and only if $\gcd(m, P_n) = 1$. For $1 < m < p_{n+1}$, $m$ is necessarily divisible by at least one prime $p_k \le p_n$, implying $\chi_n(m) = 0$. By definition, $\gcd(p_{n+1}, P_n) = 1$, so $\chi_n(p_{n+1}) = 1$. This establishes $p_{n+1}$ as the minimum integer $m > 1$ satisfying the condition.

To prove the uniqueness certified by the floor function, consider the sum 
$$
S_n = \sum_{m=1}^{2p_n} \frac{\chi_n(m)}{m}.
$$
We decompose it as:
\begin{equation*}
S_n = \chi_n(1) + \frac{\chi_n(p_{n+1})}{p_{n+1}} + \sum_{m=p_{n+1}+1}^{2p_n} \frac{\chi_n(m)}{m} = 1 + \frac{1}{p_{n+1}} + R_n.
\end{equation*}
Using Bertrand's postulate ($p_{n+1} < 2p_n$) and the fact that $p_{n+1} \ge 3$, the remainder $R_n$ satisfies:
\begin{equation*}
0 \le R_n < \sum_{m=p_{n+1}+1}^{2p_n} \frac{1}{m} < \ln\left(\frac{2p_n}{p_{n+1}}\right) < \ln 2.
\end{equation*}
Since 
$$
\frac{1}{p_{n+1}} + \ln 2 < 1
$$ 
for all $n \ge 1$, we have $1 < S_n < 2$, which implies 
$$
\lfloor S_n \rfloor = 1.
$$
This floor value confirms that in the interval $[1, 2p_n]$, no other integer $m > 1$ besides $p_{n+1}$ can contribute significantly to the sum, thereby certifying $p_{n+1}$ as the unique survivor of the sieve.
\end{proof}

This identity provides an exact analytical implementation of the Eratosthenes sieve, where the harmonic sum acts as a nonlinear threshold detector: since the ``harmonic tail'' of the remaining terms in $S_n$ is strictly bounded, $\lfloor S_n \rfloor$ isolates the first non-trivial survivor of the sieve, which is $1$. The value $p_{n+1}$ is thus identified as the smallest $m > 1$ such that $\chi_n(m)=1$. The complexity of this formula is significantly lower than the $O(2^n)$ complexity of Gandhi's inclusion-exclusion sum. The number of iterations is $O(p_n)$, which is approximately $O(n \ln n)$. With $O(\log p_n)$ for each gcd calculation, the total complexity is $O(p_n \log p_n)$, a sub-exponential growth rate.

\section{Second formula: spectral perspective and harmonic resonance}\label{sec3}

In contrast with the first formula, which is an exact arithmetic identity, the present approach proposes a spectral model inspired by Riemann's explicit formula. Here, the location of primes is interpreted as the result of constructive and destructive interference between arithmetical oscillations.

The connection between the smooth asymptotic trend and the discrete locations of primes is governed by Riemann's explicit formula for the Chebyshev function $\psi(x) = \sum_{k \le x} \Lambda(k)$:
\begin{equation*}
\psi(x) = x - \sum_{\rho} \frac{x^\rho}{\rho} - \ln(2\pi) - \frac{1}{2}\ln(1-x^{-2}),
\end{equation*}
where $\rho = \beta + i\gamma$ denotes the non-trivial zeros of $\zeta(s)$. This formula shows that prime distribution is governed by a superposition of oscillatory modes indexed by the zeros of $\zeta(s)$, which naturally motivates a spectral reconstruction. In this model, we approximate the $n$-th prime by perturbing the high-order Cipolla expansion $T(n)$ with a sum of harmonics \cite{Cipolla1902}:
\begin{equation*}
p_n \approx \left\lfloor T(n) + \alpha \sum_{k=2}^{\lfloor \sqrt{T(n)} \rfloor} \Lambda(k) \cos\left(\frac{2\pi n}{\ln k}\right) + \mathcal{R}(n) \right\rfloor,
\end{equation*}
where $T(n)$ accounts for the global logarithmic drift:
\begin{equation}\label{trois}
T(n) = n \left[ \ln n + \ln \ln n - 1 + \frac{\ln \ln n - 2}{\ln n} - \frac{(\ln \ln n)^2 - 6 \ln \ln n + 11}{2 (\ln n)^2} \right].
\end{equation}
The scaling factor $\alpha$ calibrates the resonance amplitude. The summation is restricted to $k \le \sqrt{T(n)}$, so that only the dominant spectral contributions are retained while higher-frequency oscillations are neglected. This cutoff plays a role analogous to a bandwidth limitation in signal processing: only the main harmonics shaping the global structure are preserved, whereas fine-scale fluctuations are filtered out. No claim of convergence or exactness is made; the model should be understood as a phenomenological reconstruction inspired by the structure of Riemann's explicit formula. The choice of the von Mangoldt function $\Lambda(k)$ as the weighting factor is dictated by its role as the natural detector of prime power singularities, established through the logarithmic derivative of the Riemann zeta function:
\begin{equation*}
\frac{\zeta'(s)}{\zeta(s)} = -\sum_{k=1}^{\infty} \frac{\Lambda(k)}{k^s}, \quad \text{Re}(s) > 1.
\end{equation*}

In this framework, $\Lambda(k)$ acts as an amplitude gain, ensuring the spectral sum resonates only at arithmetically relevant coordinates. The term $\cos(2\pi n / \ln k)$ functions as a periodic filter; at $x \approx p_n$, the phases of the oscillations achieve constructive interference, producing a localized density peak that effectively ``pins'' the prime position. Conversely, between primes, the phases generate destructive interference, cancelling the probability of prime occurrence.

The residual term $\mathcal{R}(n)$ accounts for the contribution of the non-trivial zeros $\rho$ of the Riemann zeta function that lie beyond the spectral cutoff, as well as the higher-order terms of the prime power fluctuations. In this phenomenological model, $\mathcal{R}(n)$ is assumed to be an oscillatory noise of order $O(\sqrt{T(n)}\ln T(n))$ under the Riemann Hypothesis.

A significant feature of this model is its sensitivity to prime powers $k = p^m$. While primary resonances occur at primes ($m=1$), the sum accounts for secondary harmonics at $p^2, p^3, \dots$, contributing a ``higher-harmonic noise'' of order $\sqrt{x}$. The residual term $\mathcal{R}(n)$ effectively absorbs these secondary nodes, allowing the floor function to isolate the integer coordinate of $p_n$ rather than a nearby prime power singularity.

\section{Third formula: ``survival dynamics'' and information entropy}\label{sec4}

This third framework is not meant to provide an exact identity for $p_n$, but a probabilistic and information-theoretic interpretation of its growth, rooted in Mertens' theorem and sieve theory. In contrast to the spectral view, the survival framework treats the sequence of primes as a dynamic process governed by the depletion of available arithmetic states.

\subsection{Mertens' theorem as a constraint}\label{subsec41}

The probability that a large integer $x$ is not divisible by any prime $p \le p_n$ is traditionally modeled by the fundamental sieve product:
\begin{equation*}
\Phi(x, p_n) = \prod_{p \le p_n} \left( 1 - \frac{1}{p} \right).
\end{equation*}
According to Mertens' third theorem, as $n \to \infty$, this product converges as follows:
\begin{equation*}
\prod_{k=1}^n \left( 1 - \frac{1}{p_k} \right) \sim \frac{e^{-\gamma}}{\ln p_n},
\end{equation*}
where the term $e^{-\gamma} \approx 0.56146$ acts as a ``density efficiency'' constant. In this framework, we interpret $e^{-\gamma}$ not merely as a limit, but as the asymptotic survival rate of integers under the pressure of the Eratosthenes sieve.

\subsection{Information entropy and the prime sieve}\label{subsec42}

The emergence of $p_n$ can be modeled as an information-theoretic event. Let $X$ be a random variable representing the primality of an integer. The uncertainty (Shannon entropy) $H(X)$ regarding the exact position of $p_n$ increases as the density of primes decays. If the probability of an integer being prime is 
$$
\mathscr{P}(x) \approx 1/\ln x,
$$
the information content (surprisal) of finding a prime at $x$ is
$$
I(x) = -\log_2(\mathscr{P}(x)) = \log_2(\ln x).
$$
The global entropy associated with the distribution up to $n$ is:
\begin{equation*}
H_n = - \sum_{k=2}^{n} \mathscr{P}(k) \ln \mathscr{P}(k) \approx \int_{2}^{n} \frac{\ln(\ln x)}{\ln x} dx.
\end{equation*}
This entropy represents the ``search cost'' required to locate the next singularity in the sequence. By integrating this resistance with the Mertens limit, $p_n$ emerges as the coordinate where the accumulated ``information debt'' of the sieve is precisely compensated by the $e^{-\gamma}$ survival constant:
\begin{equation*}
p_n \approx \left\lfloor (n \ln n) \cdot \prod_{k=2}^{n} \left( 1 + \frac{1}{k \ln k - \ln \ln k} \right) \cdot e^{-\gamma} \right\rfloor.
\end{equation*}
This expression is obtained by combining the asymptotic trend $n \ln n$ with a recursive correction factor. The product term $\prod (1 + 1/(k \ln k - \ln \ln k))$ acts as a discrete growth multiplier that compensates for the logarithmic depletion of prime density. By normalizing this growth with the Mertens constant $e^{-\gamma}$, we align the probabilistic ``survival'' probability with the actual count of the sieve. Essentially, it represents the $n$-th prime as a cumulative product of survival events where the information cost of each new prime is factored into the preceding density. This identity serves as a bridge between the discrete sieve of Eratosthenes and the continuous limit of the prime distribution, where the product encapsulates the local ``noise'' of the density fluctuations.

\subsection{Selberg's sieve and dynamics optimization}\label{subsec43}

To refine the ``survival-dynamics'' model, it is necessary to move beyond the rigid inclusion-exclusion principle. While the M\"obius filter $\chi_n(m)$ used in our first formula is theoretically exact, it is analytically ``stiff'' due to its binary nature. Selberg's sieve \cite{Selberg1947,Hooley1976,Ribenboim2011} introduces a more flexible, quadratic approach to approximate the characteristic function of primes.

\subsubsection{The quadratic filtering, information entropy and the extremal principle}\label{subsubsec431}

In Selberg's framework, the survival probability is modeled by real weights $\lambda_d$, chosen to minimize the quadratic form:
\begin{equation*}
S(x, z) = \sum_{n \le x} \left( \sum_{d|n, d < z} \lambda_d \right)^2,
\end{equation*}
subject to the constraint $\lambda_1 = 1$. This optimization problem is equivalent to finding the path of least ``arithmetic resistance,'' acting as a ``dampened'' version of the M\"obius function.

From an information-theoretic perspective, Selberg's sieve represents the optimal compression of the sieving process. By choosing weights that minimize $S(x, z)$, the remainder term in the prediction of $p_n$ is significantly reduced. This suggests an extremal principle where the sequence $\{p_n\}$ emerges as the most efficient configuration to distribute non-divisibility across the set of integers. 

\subsubsection{The sieve capacity identity and structural density}\label{subsubsec433}

The introduction of the ``sieve capacity'', denoted here as $\mathcal{C}(p_n)$, quantifies the available density of integers that escape the initial sieving stages. In Selberg’s theory, this is related to the function 
$$
V(z) = \sum_{d < z} \frac{\mu^2(d)}{f(d)},
$$
which serves as a measure of the cumulative ``breathing space'' left by the $n$ preceding primes. In this survival dynamics model, $f(d)$ acts as a structural weighting function. It quantifies the ``arithmetic resistance'' of each divisor $d$, determining how much ``sieve capacity'' is consumed by the preceding primes. We define this effective capacity as:
\begin{equation*}
\mathcal{C}(p_n) \approx \frac{1}{V(z)},
\end{equation*}
representing the proportion of integers surviving the sieve. The parameter $z$ represents the sieve level, typically optimized as $z \approx \sqrt{p_n}$ to control the growth of the remainder term $\mathcal{R}_{\text{Selberg}}$. This leads to the identity:
\begin{equation*}
p_n \approx \left\lfloor \frac{n}{\mathcal{C}(p_n)} + \mathcal{R}_{\text{Selberg}}(n) \right\rfloor.
\end{equation*}
By treating the distribution as a continuous flow through a capacity-limited filter, we model $p_{n+1}$ as the point where the arithmetic space is no longer sufficient to maintain the required survival density dictated by the preceding set $P_n$.

The remainder $\mathcal{R}_{\text{Selberg}}(n)$ represents the accumulation of local errors in the approximation of the characteristic function of primes. It is primarily governed by the parity problem in sieve theory and the density of the ``admissible'' remaining intervals, typically behaving as a sub-leading correction to the main capacity-driven trend.

\subsubsection{Capacity depletion and the twin prime conjecture}\label{subsubsec434}

The concept of sieve capacity also provides a structural explanation for the convergence of Brun's constant $B_2$ \cite{Brun1919}:
\begin{equation*}
B_2 = \sum_{p, \, p+2 \in \mathbb{P}} \left( \frac{1}{p} + \frac{1}{p+2} \right) \approx 1.90216,
\end{equation*}
where $\mathbb{P}$ denotes the ensemble of prime numbers. In our framework, this convergence is interpreted as a ``saturation of local capacity''. For a pair $(p, p+2)$ to emerge, the sieve capacity must remain locally high enough to allow two adjacent survivors. As $n \to \infty$, the global capacity $\mathcal{C}(p_n)$ decays as $1/\ln p_n$. The finite value of $B_2$ suggests that the arithmetic volume occupied by twin primes is intrinsically bounded by this logarithmic depletion. This aligns with Selberg's upper bounds, suggesting that while twin primes may exist infinitely, they do so with a vanishingly small local density compared to the general prime population.

\section{Numerical discussion and implementation limits}\label{sec5}

While the three formulas provide a robust theoretical framework, their translation into computational algorithms reveals fundamental constraints related to machine precision and the nature of the approximations.

\subsection{The precision bottleneck and certification}\label{subsec51}

The discrete sieve identity, although mathematically exact, is highly sensitive to the numerical resolution of the harmonic sum. In a standard numerical environment (such as 64-bit floating-point arithmetic), the quantity $\delta_n = S_n - 1$ acts as a high-resolution probe. Since 
$$
\delta_n = 1/p_{n+1} + R_n,
$$
the presence of the harmonic tail $R_n$ represents a noise floor that prevents a direct inversion to recover $p_{n+1}$. As $n$ increases, the difference $\delta_n$ approaches the machine epsilon ($\epsilon \approx 2.22 \times 10^{-16}$), compromising the stability of the floor function $\lfloor S_n \rfloor = 1$ used to certify the sieve's integrity. This numerical limit necessitates the use of arbitrary-precision arithmetic to distinguish the prime signal from the surrounding harmonic noise.

\subsection{Spectral resolution and Gibbs phenomenon in the second formula}\label{subsec52}

The harmonic resonance model utilizes a truncated version of the von Mangoldt oscillations. In our implementation, the summation over $k$ is limited to $\sqrt{T(n)}$ to maintain computational efficiency. For relatively small $n$, this acts as a ``low-pass filter'' that is insufficient to capture the sharp singularities of the prime distribution. The observed residuals represent an arithmetic version of the Gibbs phenomenon: without the inclusion of a larger set of non-trivial zeros of the Riemann zeta function, the spectral reconstruction provides a smooth density approximation rather than a discrete identification of $p_n$.

\subsection{Asymptotic drift of survival dynamics in the third formula}\label{subsec53}

The survival dynamics approach exhibits a systematic underestimation in the pre-asymptotic regime ($n < 100$). This is consistent with the behavior of the Mertens constant $e^{-\gamma}$, which governs the distribution at the limit $n \to \infty$. For finite $n$, the cumulative growth index does not yet fully compensate for the local density fluctuations. This drift highlights the transition between the deterministic sieve and the probabilistic baseline, where the ``information debt'' of the prime sequence is only balanced over very large intervals.

\section{Conclusion}\label{sec6}

The comparative analysis of the three models presented in this study reveals that the $n$-th prime number is not merely an isolated arithmetic value, but the result of a multi-scale interaction. By bridging discrete sieving, spectral analysis, and information theory, we can move beyond the ``random'' appearance of primes toward a structured understanding of their emergence.

The first model demonstrates that $p_{n+1}$ is exactly recoverable through a discrete gating mechanism, provided the harmonic sum is correctly filtered. The second model shifts this perspective to the frequency domain, where primes appear as points of constructive interference, effectively ``pinning'' the integer coordinates through wave synchronization. Finally, the third model provides the global growth envelope, where the search for the next prime is governed by the exhaustion of arithmetic capacity and the accumulation of information entropy.

We can consider that $p_n$ is the simultaneous product of these two complementary regimes: a global growth law dictated by cumulative information debt, and a local pinning effect determined by spectral resonance. This duality suggests that the distribution of primes is a highly correlated system where the emergence of a new singularity $p_{n+1}$ is strictly conditioned by the information density of the preceding set $P_n$.

From a broader mathematical standpoint, these frameworks transition the study of prime numbers from static identities to the language of complex systems and thermodynamics. Modeling primes as a ``Coulomb gas'' or a capacity-limited flow offers a path toward reconciling the rigid determinism of the sieve of Eratosthenes with the probabilistic beauty of the Riemann zeta function. Future work exploring the coupling between spectral noise and entropy-driven baselines may offer deeper insights into the transition from asymptotic trends to the discrete arithmetic reality of the primes.

\section*{Appendix A: probabilistic proof of Gandhi's formula}

Let $(p_n)_{n \in \mathbb{N}^*}$ be the sequence of prime numbers. In 1971, J.M. Gandhi proved that for every $n \in \mathbb{N}^*$, the prime number $p_{n+1}$ is the unique integer $m$ such that:
\begin{equation}\label{eq:gandhi_bounds}
1 < 2^m \left( \frac{1}{2} + \sum_{k=1}^n (-1)^k \sum_{1 \le i_1 < \dots < i_k \le n} \frac{1}{2^{p_{i_1} \dots p_{i_k}} - 1} \right) < 2.
\end{equation}
The objective of this appendix is to establish this property using probabilistic reasoning, following Ref. \cite{Meyre}. We consider the random drawing of a non-zero integer according to the geometric distribution with parameter $1/2$. We work in the probability space $(\Omega, \mathcal{F}, P)$, with $\Omega = \mathbb{N}^*$, $\mathcal{F} = \mathcal{P}(\mathbb{N}^*)$, and $\mathscr{P}$ is the geometric law of parameter 1/2, this latter distribution being associated with the mass function $\mathfrak{p}: \mathbb{N}^* \to \mathbb{R}_+$ defined by:
$$
\forall k \in \mathbb{N}^*, \mathfrak{p}_k = \frac{1}{2^k}.
$$

Let $A_d$ be the event ``the result is divisible by $d$'', such that $A_d = d\mathbb{N}^*$ and
$$
\mathscr{P}(A_d) = \sum_{k=1}^\infty \frac{1}{2^{kd}}.
$$

By using the sum of a geometric series (here, with common ratio $1/2^d$), we deduce that the sought probability is:
\begin{equation}\label{dix}
\mathscr{P}(A_d) = \frac{1}{2^d - 1}.
\end{equation}
Let $B_n$ be the event ``the result is coprime with $p_1 \times \dots \times p_n$'', so that $\pi_n = \mathscr{P}(B_n)$. Since $p_1, \dots, p_n$ are prime numbers, $\omega \in B_n$ if $\omega$ is not divisible by any of the numbers $p_1, \dots, p_n$. In other words,
$$
B_n = \bigcap_{i=1}^n A_{p_i}^c
$$
or, by passing to the complement,
$$
B_n^c = \bigcup_{i=1}^n A_{p_i}.
$$
By using Poincar\'e's formula (inclusion-exclusion principle), we deduce, by passing to the probability $\mathscr{P}$:
$$
1 - \pi_n = \sum_{k=1}^n (-1)^{k-1} \sum_{1 \le i_1 < \dots < i_k \le n} \mathscr{P}(A_{p_{i_1}} \cap \dots \cap A_{p_{i_k}}).
$$
Since $p_{i_1}, \dots, p_{i_k}$ are always prime numbers, we have the following equality for all $k \in \{1, \dots, n\}$ and all $1 \le i_1 < \dots < i_k \le n$:
$$
A_{p_{i_1}} \cap \dots \dots \cap A_{p_{i_k}} = A_{p_{i_1} \dots p_{i_k}}.
$$
We deduce, using equality (\ref{dix}):
$$
1 - \pi_n = \sum_{k=1}^n (-1)^{k-1} \sum_{1 \le i_1 < \dots < i_k \le n} \frac{1}{2^{p_{i_1} \dots p_{i_k}} - 1},
$$
which immediately gives the sought equality:
\begin{equation}\label{dix-huit}
\pi_n = 1 + \sum_{k=1}^n (-1)^k \sum_{1 \le i_1 < \dots < i_k \le n} \frac{1}{2^{p_{i_1} \dots p_{i_k}} - 1}.
\end{equation}
The smallest integer strictly greater than 1 that is coprime with the product $p_1 \times \dots \times p_n$ is none other than $p_{n+1}$, as a moment's reflection shows. Furthermore, since $p_{n+1}$ is odd, it is not an element of $B_n$. We deduce that $B_n = \{1, p_{n+1}\} \cup C_n$, with $\min C_n \ge p_{n+1} + 2$. By passing to the probability $\mathscr{P}$, we have:
\begin{equation}\label{eq:pi_n}
\pi_n = \frac{1}{2} + \frac{1}{2^{p_{n+1}}} + \mathscr{P}(C_n),
\end{equation}
with the inequalities:
$$
0 < \mathscr{P}(C_n) < \sum_{k=p_{n+1}+2}^\infty \frac{1}{2^k} = \frac{1}{2^{p_{n+1}+1}},
$$
which indeed leads to the sought formula (\ref{eq:pi_n}). \\
By comparing equality (\ref{dix-huit}) and formula (\ref{eq:pi_n}), we obtain:
$$
2^{p_{n+1}} \left( \frac{1}{2} + \sum_{k=1}^n (-1)^k \sum_{1 \le i_1 < \dots < i_k \le n} \frac{1}{2^{p_{i_1} \dots p_{i_k}} - 1} \right) - 1 = \theta_n
$$
which allows us to conclude, since $0 < \theta_n < 1/2 < 1$.

By taking the base-2 logarithm, denoted $\log_2$, in the inequality range (\ref{eq:gandhi_bounds}), we obtain:
$$
0 < m + \log_2 \left( \frac{1}{2} + \sum_{k=1}^n (-1)^k \sum_{1 \le i_1 < \dots < i_k \le n} \frac{1}{2^{p_{i_1} \dots p_{i_k}} - 1} \right) < 1.
$$
We easily deduce the following bounds:
$$
m < 1 - \log_2 \left( \frac{1}{2} + \sum_{k=1}^n (-1)^k \sum_{1 \le i_1 < \dots < i_k \le n} \frac{1}{2^{p_{i_1} \dots p_{i_k}} - 1} \right) < m + 1
$$
which implies the equality:
$$
m = \left\lfloor 1 - \log_2 \left( \frac{1}{2} + \sum_{k=1}^n (-1)^k \sum_{1 \le i_1 < \dots < i_k \le n} \frac{1}{2^{p_{i_1} \dots p_{i_k}} - 1} \right) \right\rfloor.
$$
We have thus proven both the uniqueness of $m$ and, according to the previous question, established the equality:
$$
\forall n \in \mathbb{N}^*, \quad p_{n+1} = \left\lfloor 1 - \log_2 \left( \frac{1}{2} + \sum_{k=1}^n (-1)^k \sum_{1 \le i_1 < \dots < i_k \le n} \frac{1}{2^{p_{i_1} \dots p_{i_k}} - 1} \right) \right\rfloor.
$$

\end{document}